\documentclass[letterpaper,12pt]{amsart}

\usepackage[divide={1in,*,1in}]{geometry}
\usepackage{amsfonts,amssymb}
\usepackage{mathtools}
\usepackage{stmaryrd}
\usepackage{tabularx}
\usepackage{hhline}

\usepackage{pstricks,pst-node}
\usepackage{centerpict}
\usepackage{diagps,diagps-arrows}
\usepackage{movies}
\usepackage{mywrap}

\usepackage[
		colorlinks,
		pdfauthor={Krzysztof K. Putyra and Alexander N. Shumakovitch},
		pdftitle={Knot invariants arising from homological operations on Khovanov homology}]{hyperref}
\usepackage[all]{hypcap}		
\usepackage[savepos]{zref}
\usepackage[tracking=off]{comments}

\newtheorem{theorem}{Theorem}[section]
\newtheorem{lemma}[theorem]{Lemma}
\newtheorem{conjecture}[theorem]{Conjecture}
\newtheorem{proposition}[theorem]{Proposition}
\newtheorem{corollary}[theorem]{Corollary}

\theoremstyle{definition}

\newtheorem{definition}[theorem]{Definition}
\newtheorem{example}[theorem]{Example}
\newtheorem{remark}[theorem]{Remark}
\newtheorem{observation}[theorem]{Observation}

\newcount\repeatedID
\def\repeattheorem#1#2{%
	\theoremstyle{plain}
	\newtheorem*{RepeatedTheorem\the\repeatedID}{#1~\ref{#2}}%
	\begin{RepeatedTheorem\the\repeatedID}%
}
\def\endrepeattheorem{%
	\end{RepeatedTheorem\the\repeatedID}%
	\global\advance\repeatedID 1\relax
}
 
\definecolor{internalLink}{rgb}{0.5,0,0}
\definecolor{citeLink}{rgb}{0,0.5,0}
\definecolor{urlLink}{rgb}{0,0,0.5}

\hypersetup{linkcolor=internalLink}
\hypersetup{citecolor=citeLink}
\hypersetup{urlcolor=urlLink}

\newcommand{\singlespace}{\renewcommand{\baselinestretch}{1.10}\selectfont}
\newcommand{\oneandhalfspace}{\renewcommand{\baselinestretch}{1.25}\selectfont}

\DeclareMathOperator{\id}{id}		
\DeclareMathOperator{\rk}{rk}		

\let\oddParameter\xi

\def\Z{\mathbb{Z}}
\def\Zev{\Z_e}
\def\Zodd{\Z_o}
\def\Zunfd{\Z_{\oddParameter}}
\def\ZmodTwo{\Z_2}

\makeatletter

\def\scalarsLong{\@ifstar
	{\Z[\permMM,\permSS,\permMS^{\pm1}]/(\permMM^2{=}\permSS^2{=}1)}%
	{\Z[\permMM,\permSS,\permMS^{\pm1}]/(\permMM^2=\permSS^2=1)}}
\def\ZevLong{\@ifstar
	{\Zunfd/(\oddParameter{-}1)}%
	{\Zunfd/(\oddParameter-1)}}
\def\ZoddLong{\@ifstar
	{\Zunfd/(\oddParameter{+}1)}%
	{\Zunfd/(\oddParameter+1)}}
\def\ZunfdLong{\@ifstar
	{\Z[\oddParameter]/(\oddParameter^2{-}1)}%
	{\Z[\oddParameter]/(\oddParameter^2-1)}}

\makeatother

\newcommand*{\KhCom}{\mathcal C}
\newcommand*{\EKhCom}{\KhCom_e}
\newcommand*{\OKhCom}{\KhCom_o}

\newcommand*{\UKhCom}{\KhCom_{\oddParameter}}
\newcommand*{\KhComModTwo}{\KhCom_{\ZmodTwo}}

\newcommand*{\RedKhCom}{\overline{\KhCom}}

\newcommand*{\RedUKhCom}{\RedKhCom_{\oddParameter}}

\newcommand*{\Kh}{\mathcal{H}}
\newcommand*{\EKh}{\Kh_e}
\newcommand*{\OKh}{\Kh_o}
\newcommand*{\BKh}{\Kh_{\oplus}}
\newcommand*{\UKh}{\Kh_{\oddParameter}}
\newcommand*{\KhModTwo}{\Kh_{\ZmodTwo}}

\newcommand*{\RedKh}{\overline{\Kh}}
\newcommand*{\RedEKh}{\RedKh_e}
\newcommand*{\RedOKh}{\RedKh_o}
\newcommand*{\RedUKh}{\RedKh_{\oddParameter}}
\newcommand*{\RedKhModTwo}{\RedKh_{\ZmodTwo}}

\def\diffEv{d_e}
\def\diffOdd{d_o}
\def\diffUnfd{d_{\oddParameter}}
\def\diffModTwo{d}

\def\BockEv{\beta_{e}}
\def\BockOdd{\beta_{o}}
\def\BockMixed{\beta}
\def\BockEvOdd{\varphi_{eo}}
\def\BockOddEv{\varphi_{oe}}
\def\BockSqEv{\theta_{e}}
\def\BockSqOdd{\theta_{o}}

\def\BockEvOddDS{\expandafter\tilde\BockEvOdd}
\def\BockOddEvDS{\expandafter\tilde\BockOddEv}

\def\RedBockEv{\bar\beta_{e}}
\def\RedBockOdd{\bar\beta_{o}}
\def\RedBockMixed{\bar\beta}
\def\RedBockEvOdd{\bar\varphi_{eo}}
\def\RedBockOddEv{\bar\varphi_{oe}}

\def\RedBockSqOdd{\bar\theta_{o}}

\def\Sq{\mathit{Sq}}

\definecolor{bockev}{rgb}{0,0,1}
\definecolor{bockodd}{rgb}{1,0,0}

\def\bockevarrow#1#2{%
	\diagline[linestyle=none]{-}{#1}{#2}\diagaput*{$\scriptscriptstyle\BockEv$}%
	\diagline[nodesep=-2pt,linecolor=bockev]{->}{#1}{#2}%
}
\def\bockoddarrow#1#2{%
	\diagline[linestyle=none]{-}{#1}{#2}\diagaput*{$\scriptscriptstyle\BockOdd$}%
	\diagline[nodesep=-2pt,linecolor=bockodd]{->}{#1}{#2}%
}

\newcommand*\Ext{\mathrm{Ext}}		
\newcommand*\Tor{\mathrm{Tor}}		

\makeatletter

\def\dertensor{\mskip\thinmuskip{\hat\tensor}\mskip\thinmuskip}
\def\tensor{\@ifnextchar[\tensor@over\otimes}
\def\tensor@over[#1]{\mskip\thinmuskip{\underset{\mathclap{#1}}{\otimes}}\mskip\thinmuskip}

\makeatother

\newcommand*\LieSL{\mathfrak{sl}}

\def\quot#1{`#1'}

\makeatletter

\def\arXiv{\@ifstar\arXiv@@\arXiv@}
\def\arXiv@#1{\href{http://front.math.ucdavis.edu/#1}{arXiv:#1}}
\def\arXiv@@#1#2{\href{http://front.math.ucdavis.edu/#2}{arXiv:#1}}

\makeatother

\allowdisplaybreaks

\begin{document}

\oneandhalfspace


\title[Knot Invariants from Homological Operations]{Knot Invariants Arising\\
from Homological Operations\\ on Khovanov Homology}

\author{Krzysztof K.\ Putyra}
\address{ETH~Z\"urich, Institute for Theoretical Studies, Z\"urich, Switzerland}
\email{krzysztof.putyra@eth-its.ethz.ch}
\thanks{KKP is supported by the~NCCR SwissMAP founded by the~Swiss National Science Foundation}

\author{Alexander N.\ Shumakovitch}
\address{Department of Mathematics, The George Washington University,
Washington DC, U.S.A.}
\email{Shurik@gwu.edu}
\thanks{A.Sh.\ is partially supported by a Simons Collaboration Grant for
Mathematicians~\#279867}

\begin{abstract}
	We construct an~algebra of non-trivial homological operations on Khovanov homology with coefficients in $\ZmodTwo$ generated by two Bockstein operations. We use the~unified Khovanov homology theory developed by the~first author to lift this algebra to integral Khovanov homology. We conjecture that these two algebras are infinite and present evidence in support of our conjectures. Finally, we list examples of knots that have the same even and odd Khovanov homology, but different  actions of these homological operations. This confirms that the~unified theory is a~finer knot invariant than the even and odd Khovanov homology combined. The~case of reduced Khovanov homology is also considered.
\end{abstract}

\maketitle

\renewcommand{\labelenumi}{\theenumi.}

\section{Introduction}\label{sec:intro}

Throughout this paper, $L$ will denote a link in $\mathbb R^3$ and $D$ its planar diagram. There are two integral versions of $\LieSL_2$ link homology of $L$: the ordinary Khovanov homology $\EKh(L)$~\cite{KhHom}, which we refer to is this paper as \emph{even}, and the odd Khovanov homology $\OKh(L)$~\cite{OddKhHom}. Despite being different over integers~\cite{ShumComp}, even and odd Khovanov homology theories agree modulo $2$. We denote the resulting homology by $\KhModTwo(L)$. Differentials in the even and odd Khovanov chain complexes induce two Bockstein connecting homomorphisms on $\KhModTwo(L)$ that correspond to the~short exact sequence of coefficients $0\to\Z_2\to\Z_4\to\Z_2\to0$. We denote them by $\BockEv$ and $\BockOdd$ respectively:
\begin{equation}
	\BockEv,\BockOdd\colon \KhModTwo^i(L)\to \KhModTwo^{i+1}(L).
\end{equation}

Ranks of each of $\BockEv$ and $\BockOdd$ can be easily recovered from the~integral homology: they are equal to the number of invariant factors isomorphic to $\Z_2$ in the corresponding homology group, see~\cite[Proposition~3E.3]{Hatcher}. On the other hand, $\BockEv$ and $\BockOdd$ are as algebraically independent as possible: they not only do not commute, but all their alternating compositions are nontrivial. More precisely, let $\BockMixed:=\BockEv+\BockOdd$. Since $\BockEv^2=\BockOdd^2=0$, we have that $\BockMixed\BockEv=\BockOdd\BockEv$ and $\BockMixed\BockEv=\BockEv\BockOdd$. Hence, the only nontrivial compositions of Bockstein homomorphisms of length $n$ are $\BockMixed^{n-1} \BockEv$ and $\BockMixed^{n-1} \BockOdd$, whereas $\BockMixed^n$ is their sum.

The~even Khovanov homology is multiplicative with respect to disjoint unions of links, $\EKh(L\sqcup L') \cong \EKh(L) \dertensor \EKh(L')$ \cite{KhHom} (here $\dertensor$ stands for the~derived tensor product), which implies primitivity of the~even Bockstein homomorphism $\BockEv$. The~first author has recently proved the~multiplicativity of the~odd homology \cite{KhHomTensor}, which implies primitivity of $\BockOdd$. Similar formulas hold for connected sums of knots, in which case one considers a~derived tensor product over the~algebra assigned to a~circle \cite{KhPatterns,KhHomTensor}. This allowed us to construct in Example~\ref{ex:4-bocks} a~knot with 20 crossings for which the~alternating compositions of four Bockstein homomorphisms are nontrivial.

\begin{repeattheorem}{Conjecture}{conj:Bocksteins-algebra}
	The~operations $\BockMixed^n$, $\BockMixed^{n-1}\BockEv$, and $\BockMixed^{n-1}\BockOdd$ are nonzero and pairwise different.
\end{repeattheorem}

The~two Bockstein homomorphisms generate three degree 2 operations: $\BockEv\BockOdd$, $\BockOdd\BockEv$, and the~commutator $\BockMixed^2=\BockEv\BockOdd+\BockOdd\BockEv$. It appears that none of them is the~second Steenrod square $\Sq^2$ as defined in \cite{KhSq}, because the~ranks do not match. Moreover $\BockMixed^2$ is primitive, while $\Sq^2$ is not (see~\cite[Section 4.L]{Hatcher}).

\begin{corollary}
	The~degree 2 operations $\BockEv\BockOdd$, $\BockOdd\BockEv$, and $\BockMixed^2$ are different from the~second Steenrod square $\Sq^2$.
\end{corollary}

Another surprising fact is the~existence of integral lifts of both Bockstein operations. They are constructed using a~link homology $\UKh(L)$ that is defined over a~ring $\Zunfd:=\ZunfdLong$ and unifies the~even and odd Khovanov homology theories. $\EKh(L)$ and $\OKh(L)$ can be recovered from $\UKh(L)$ by taking coefficients in certain modules over $\Zunfd$ \cite{ChCob}. Namely, there are isomorphisms of graded abelian groups
\begin{equation}\label{eq:EKh-and-OKh-from-UKh}
	\EKh(L)\cong\UKh(L;\Zev)
	\qquad\text{and}\qquad
	\OKh(L)\cong\UKh(L;\Zodd),
\end{equation}
where $\Zev:=\ZevLong$ and $\Zodd:=\ZoddLong$ are $\Zunfd$--modules on which $\xi$ acts as identity or negation respectively. Both $\Zev$ and $\Zodd$ are isomorphic to $\Z$ as abelian groups.

It is not immediately clear whether $\UKh(L)$ is a stronger invariant than $\EKh(L)$ and $\OKh(L)$ together. The authors are not aware of any software package that can compute the~unified homology explicitly. This might be due to the complexity of classification of modules over the~ring $\Zunfd$. The~results of this paper indicate that $\UKh(L)$ is actually a~finer invariant than the~even and odd homology combined.

Let $D$ be a~diagram of a~link $L$ in $\mathbb R^3$. Denote by $\EKhCom(D)$, $\OKhCom(D)$, and $\UKhCom(D)$ the~Khovanov chain complexes for the~even, odd, and unified homology, respectively. We recall in Section~\ref{sec:pullback} that $\UKhCom(D)$ can be considered as an~extension of complexes $\EKhCom(D)$ and $\OKhCom(D)$ in two different ways. The~two short exact sequences
\begin{gather}
	0\to \OKhCom(D) \to \UKhCom(D) \to \EKhCom(D) \to 0 \\
	0\to \EKhCom(D) \to \UKhCom(D) \to \OKhCom(D) \to 0
\end{gather}
induce connecting homomorphisms $\BockEvOdd\colon \EKh^i(L) \to<1em> \OKh^{i+1}(L)$ and $\BockOddEv\colon \OKh^i(L) \to<1em> \EKh^{i+1}(L)$. In Section~\ref{sec:oper-integral} we show that $\BockEvOdd$ and $\BockOddEv$ are, in fact, Bockstein homomorphisms corresponding to certain short exact sequences of $\Zunfd$-modules. Moreover, they are integral lifts of the~Bockstein homomorphisms $\BockOdd$ and $\BockEv$, respectively.

\begin{repeattheorem}{Proposition}{prop:lifts-of-bocks}
	Given a~link $L$ there are commuting squares
	\begin{equation*}
		\tag{\ref{diag:lifts-of-bocks}}
		\begin{diagps}(0em,-0.8ex)(24em,12.25ex)
			\square<8em,10ex>[%
				\EKh^i(L)`\OKh^{i+1}(L)`\KhModTwo^i(L)`\KhModTwo^{i+1}(L);
				\BockEvOdd```\BockOdd
			]
			\square(16em,0ex)<8em,10ex>[%
				\OKh^i(L)`\EKh^{i+1}(L)`\KhModTwo^i(L)`\KhModTwo^{i+1}(L);
				\BockOddEv```\BockEv
			]
		\end{diagps}
	\end{equation*}
\end{repeattheorem}

Similarly to the~Bockstein operations over $\ZmodTwo$, we expect that the~alternating compositions of the integral operations do not vanish, see Conjecture~\ref{conj:int-Bockstein-algebra}.

Computer-based calculations reveal pairs of knots that have the same homology groups (both odd and even), but different actions of the~integral operation, see Section~\ref{sec:computation}. This implies that the~unified link homology $\UKh(L)$ is a~stronger invariant than $\EKh(L)\oplus\OKh(L)$. In particular, $\UKhCom(L)$ is a~nontrivial extension of $\EKhCom(L)$ and $\OKhCom(L)$.

\begin{remark}
	Most of the~constructions described in this paper can be applied to arbitrary complexes over the~ring $\Zunfd$, not necessarily those arising from links. Indeed, with every such a~complex $C$ we can associate its \quot{even} and \quot{odd} integral versions, $C_e:=C\tensor\Zev$ and $C_o:=C\tensor\Zodd$, that have the same reductions modulo two $C_{\Z_2}:=C_e\tensor\ZmodTwo\cong C_o\tensor\ZmodTwo$. The~differentials in $C_e$ and $C_o$ induce two Bockstein homomorphisms on $\Kh(C_{\Z_2};\Z_2)$, each admitting an~integral lift.
\end{remark}

\subsection*{Outline}
This paper is organized as follows. We begin with a~pullback description of the~ring $\Zunfd$ and the~unified Khovanov homology $\UKh(L)$. It provides a~neat way to see the~even and odd Khovanov complexes as subcomplexes and quotient complexes of $\UKhCom(L)$ at the~same time. Section~\ref{sec:oper-mod-2} describes the~construction of Bockstein operations for homology with coefficients in $\ZmodTwo$. We present examples of knots for which compositions of up to four Bocksteins are nontrivial to support Conjecture~\ref{conj:Bocksteins-algebra}. Integral Bockstein operations and the integral analog of Conjecture~\ref{conj:Bocksteins-algebra} are treated in Section~\ref{sec:oper-integral}. Section~\ref{sec:reduced} contains a~brief discussion of the~case of the~reduced Khovanov homology. The results of computer-based calculation are presented in Section~\ref{sec:computation}.

\begin{table}[ht]
\begin{tabular}{l|cccc}
	\hline
	& \parbox{2cm}{\centering even\vphantom{d}}
	& \parbox{2cm}{\centering odd}
	& \parbox{2cm}{\centering unified}
	& \parbox{2cm}{\centering\singlespace reduction\\[-0.5ex] modulo 2}\rule[-2.2ex]{0pt}{5.2ex} \\
	\hline
	chain complex
			& $\EKhCom(D)$ & $\OKhCom(D)$ & $\UKhCom(D)$ & $\KhComModTwo(D)$ \\
	differential
			& $\diffEv$    & $\diffOdd$   & $\diffUnfd$  & $\diffModTwo$     \\
	homology
			& $\EKh(L)$    & $\OKh(L)$    & $\UKh(L)$    & $\KhModTwo(L)$    \\
	\hline
\end{tabular}
\bigskip
\caption{Notations for different versions of the Khovanov chain complexes and their homology for a diagram $D$ of a link $L$.}
\label{tab:notations}
\end{table}

\subsection*{Summary of notation}
In this paper, for every diagram $D$ of a~link $L$ we assign a~number of different chain complexes and their homology. The notation used is listed in Table~\ref{tab:notations}. We also define several homological operations on different versions of Khovanov homology. For convenience, we list pages on which they are introduced  in Table~\ref{tab:oper-list}.

\begin{table}[ht]
\begin{align*}
  \BockEv&\colon\KhModTwo^i(L)\to\KhModTwo^{i+1}(L) && \text{even Bockstein homomorphism} && \text{p.~\pageref{def:even-bock}} \\
	\BockOdd&\colon\KhModTwo^i(L)\to\KhModTwo^{i+1}(L) && \text{odd Bockstein homomorphism}  && \text{p.~\pageref{def:odd-bock}}  \\
	\BockMixed&\colon\KhModTwo^i(L)\to\KhModTwo^{i+1}(L) && \text{mixed Bockstein homomorphism}&& \text{p.~\pageref{def:mixed-bock}}\\
	\BockOddEv&\colon\OKh^i(L)\to\EKh^{i+1}(L) && \text{integral even Bockstein homomorphism} && \text{p.~\pageref{def:odd-ev-bock}}\\
	\BockEvOdd&\colon\EKh^i(L)\to\OKh^{i+1}(L) && \text{integral odd Bockstein homomorphism} && \text{p.~\pageref{def:ev-odd-bock}}\\
	 \BockSqEv&\colon\EKh^i(L)\to\EKh^{i+2}(L) && \text{integral even degree 2 operation}&& \text{p.~\pageref{def:ev-sq-bock}} \\
	\BockSqOdd&\colon\OKh^i(L)\to\OKh^{i+2}(L) && \text{integral odd degree 2 operation} && \text{p.~\pageref{def:odd-sq-bock}}
\end{align*}
\caption{List of homological operations together with pages on which they are introduced.}
\label{tab:oper-list}
\end{table}

\subsection*{Acknowledgements}
We are grateful to Mikhail Khovanov for numerous stimulating discussions. The first author is also thankful to Stefan Friedl for his insightful remarks during the SwissKnots~2011 conference.

\section{A~pullback description of the~unified homology}\label{sec:pullback}

\wrapfigure[r]<0>{\begin{diagps}(-4.5em,-7ex)(4.7em,9.5ex)
	\node pi(0,7ex)[\Zunfd]
	\node even(-3.5em,0ex)[\Zev]
	\node odd(3.5em,0ex)[\Zodd]
	\node Z2(0,-7ex)[\ZmodTwo]
	\arrow|b|{->>}[pi`even;\oddParameter\mapsto<1em> 1]		\arrow{->>}[pi`odd;\oddParameter\mapsto<1em>-1]
	\arrow{->>}[even`Z2;]		\arrow{->>}[odd`Z2;]
\end{diagps}}
As mentioned in the~introduction, $\Zunfd=\ZunfdLong$ is the~universal ring of coefficients in our framework. The~rings $\Zev$ and $\Zodd$ arise as quotients of $\Zunfd$, and both project to $\ZmodTwo$ in a~unique way. Notice, that $\ZmodTwo$ admits a~unique $\Zunfd$--module structure, as $\oddParameter$ is invertible. Thence we obtain a~commuting square diagram shown to the~right. In fact, it is a~pullback square in the~category of rings.

\begin{lemma}
	There is an~isomorphism of rings $\Zunfd\cong\{(a,b)\in\Z^2\ |\ a\equiv b\mod 2\}$, such that the~projections on the~first and on the~second factors are exactly $\Zev$ and $\Zodd$.
\end{lemma}
\begin{proof}
	The~desired isomorphism maps $1$ to $(1,1)$ and $\oddParameter$ to $(1,-1)$. This map is injective, since $(a+b,a-b)=(0,0)$ implies $a=b=0$,
	and surjective, as $(a,b)$ with $a\equiv b\mod 2$ is an~image of $\tfrac{a+b}{2} + \tfrac{a-b}{2}\oddParameter$. To finish the~proof, notice that the~action of $\oddParameter$ preserves the~first factor, but negates the~second.
\end{proof}

\begin{corollary}
	Given a~link diagram $D$, the~chain complex $\UKhCom(D)$ is a~pullback of the~even and odd Khovanov complexes over their reductions modulo 2. In particular, it is a~subcomplex of the~direct sum $\EKhCom(D)\oplus\OKhCom(D)$.
\end{corollary}
\begin{proof}
	The~chain complex $\UKhCom(D)$ is a~sequence of free $\Zunfd$--modules, so that the~functor $M\mapsto\UKhCom(D)\tensor M$, where $M$ runs over $\Zunfd$--modules, is exact. In particular, it preserves pullbacks.
\end{proof}

According to the~corollary above, we can regard $\UKh(L)$ as a~derived pullback of the~diagram
\begin{equation}
	\EKh(L)\to/->/\KhModTwo(L)\to/<-/\OKh(L).
\end{equation}
Furthermore, the~kernel of the~projection $\UKhCom(D)\to\EKhCom(D)$ is the~subcomplex of $\OKhCom(D)$ formed by elements divisible by 2, which is isomorphic to the~odd Khovanov complex. Likewise, the~kernel of $\UKhCom(D)\to\OKhCom(D)$ is isomorphic to $\EKhCom(D)$. Hence, there are short exact sequences
\begin{gather}
\label{ses:Kh-odd-pi-even}
	0\to\OKhCom(D)\to\UKhCom(D)\to\EKhCom(D)\to 0,\mathrlap{\text{ and}}\\
\label{ses:Kh-even-pi-odd}
	0\to\EKhCom(D)\to\UKhCom(D)\to\OKhCom(D)\to 0,
\end{gather}
so that $\UKhCom(D)$ is an~extension between the~two chain complexes. We shall show in the~subsequent sections that its homology is a~stronger invariant than both even and odd homology together. However, the~difference is very subtle.

\begin{proposition}\label{prop:inv-of-unfd}
	Let $f\colon\EKhCom(L)\to\EKhCom(L')$ and $g\colon\OKhCom(L)\to\OKhCom(L')$ be quasi-isomorphisms\footnote{
		A~chain map is a~\emph{quasi-isomorphism} if it induces an~isomorphism on homology.
	}
	that agree modulo 2. Then $\UKh(L)\cong\UKh(L')$.
\end{proposition}
\begin{proof}
	The~pullback $(f,g)$ of the~chain maps $f$ and $g$ is the~desired quasi-isomorphism, which follows from the~5--lemma applied to the~exact sequence \eqref{ses:Kh-odd-pi-even}.
\end{proof}

\section{Bockstein operations in Khovanov homology}\label{sec:oper-mod-2}

Since the~odd and even differentials agree modulo 2, there are at least three
Bockstein operations on $\KhModTwo(L)$:
\begin{enumerate}

	\item\label{def:even-bock}
	the~\emph{even Bockstein}, $\BockEv\colon\KhModTwo^i(L)\to\KhModTwo^{i+1}(L)$, $\BockEv[x] = \left[\tfrac{1}{2}\diffEv x\right]$,
	
	\item\label{def:odd-bock}
	the~\emph{odd Bockstein}, $\BockOdd\colon\KhModTwo^i(L)\to\KhModTwo^{i+1}(L)$, $\BockOdd[x] = \left[\tfrac{1}{2}\diffOdd x\right]$, and
	
	\item\label{def:mixed-bock}
	the~\emph{mixed Bockstein}, $\BockMixed:=\BockEv+\BockOdd$.

\end{enumerate}
The~last one arises from the~short exact sequence of coefficients
\begin{equation}
	0\to\ZmodTwo\to\ZmodTwo[\oddParameter]/(\oddParameter^2-1)\to\ZmodTwo\to 0,
\end{equation}
where the~inclusion of $\ZmodTwo$ maps $1$ into $1+\oddParameter$. Indeed,
from the~pullback description of $\UKhCom(L)$ we have $\diffUnfd =
\frac{1}{2}(\diffEv+\diffOdd) + \frac{\oddParameter}{2}(\diffEv-\diffOdd) =
\frac{1+\oddParameter}{2}(\diffEv+\diffOdd)$, so that $\BockMixed[x] =
\left[\frac{1}{2}\diffUnfd x\right] = \left[\frac{1}{2}\diffEv x +
\frac{1}{2}\diffOdd x\right]$. To see that there is nothing more, use the~free
resolution
\begin{equation}
  \rule{0pt}{4ex}
	0 \to/<-/\ZmodTwo
	  \to/<-/\Zunfd
		\to/<-/^{\begin{bsmallmatrix}1+\oddParameter & 1-\oddParameter\end{bsmallmatrix}} \Zunfd^2
		\to/<-/^{\begin{bsmallmatrix}1-\oddParameter & 0 \\ 0 & 1+\oddParameter\end{bsmallmatrix}} \Zunfd^2
		\to/<-/^{\begin{bsmallmatrix}1+\oddParameter & 0 \\ 0 & 1-\oddParameter\end{bsmallmatrix}} \Zunfd^2
		\to/<-/\ldots
\end{equation}
to compute $\Ext^1_{\Zunfd}(\ZmodTwo,\ZmodTwo) \cong \ZmodTwo\oplus\ZmodTwo$. In particular, the~sum of any two of the~three Bocksteins results in the~third.

Choose a~chain complex $C$ with a~differential of degree $1$. Due to
the~Universal Coefficient Theorem, the~homology of $C$ with coefficients in
$\ZmodTwo$ is given as the~direct sum $H^i(C,\ZmodTwo)\cong
H^i(C)\tensor\ZmodTwo \oplus \Tor(H^{i+1}(C),\ZmodTwo)$. Let
$u\in\Tor(H^{i+1}(C),\ZmodTwo)$ be an element of order~$2$ that arises from an
invariant factor of $H^{i+1}(C)$ that is isomorphic to $\Z_2$. Then
the~Bockstein homomorphism on $H^i(C,\ZmodTwo)$ pairs $u$ with the~modulo 2
reduction of the~homology class that $u$ comes from. It sends every element 
of $H^i(C,\ZmodTwo)$ that does not arise in such a way to zero,
see~\cite[Proposition~3E.3]{Hatcher}.

\begin{example}\label{ex:odd-vs-even-Bock}
The~odd Khovanov homology of alternating links has no torsion~\cite{OddKhHom}, so that the~odd Bockstein is trivial. On the~other hand, the~even Khovanov homology of any torus knot $T_{2,n}$ contains a~$\ZmodTwo$ summand~\cite{KhHom}, and the~even Bockstein does not vanish.
\end{example}

The~following result follows directly from the~definition of the~operations.

\begin{proposition}\label{prop:three-operations}
	The~odd and even Bockstein homomorphisms are differentials, i.e.\ $\BockEv^2=\BockOdd^2 = 0$ and $\BockMixed^2=[\BockEv,\BockOdd]$ is their commutator.
\end{proposition}

Thence, there are at most three nontrivial operations in each degree that are generated by the~Bocksteins: the~alternating compositions $\smash{\underbrace{\cdots\BockOdd\BockEv}_n = \BockMixed^{n-1}\BockEv}$ and $\smash{\underbrace{\cdots\BockEv\BockOdd}_n = \BockMixed^{n-1}\BockOdd}$, together with their sum $\BockMixed^n$.

\begin{figure}[t]
	\centering
	\begingroup
\renewcommand\arraystretch{1}%
\def\Z{$\mathbb Z$}%
\def\F{$\mathbb Z_2$}%
\def\P{\phantom{\F}}%
\begin{tabular}{||r||c|c|c|c|c|c||}
	\hhline{|=======|}
	$\EKh$& 0& 1& 2& 3& 4& 5\\ \hhline{#=#=|=|=|=|=|=#}
	17 &\P&\P&\P&  &\P&\Z\\ \hline
	15 &  &  &  &  &  &\Z\\ \hline
	13 &  &  &  &\Z&\Z&\P\\ \hline
	11 &  &  &  &\F&\Z&  \\ \hline
	 9 &  &  &\Z&  &  &  \\ \hline
	 7 &\Z&  &  &  &  &  \\ \hline
	 5 &\Z&  &  &  &  &  \\
	\hhline{|=======|}
\end{tabular}\endgroup

	\qquad
	\begingroup
\renewcommand\arraystretch{1}%
\def\z{$\mathbb Z$}%
\def\F{$\mathbb Z_2$}%
\def\P{\phantom{\F}}%
\begin{tabular}{||r||c|c|c|c|c|c||}
	\hhline{|=======|}
	$\OKh$& 0& 1& 2& 3& 4&  5   \\ \hhline{#=#=|=|=|=|=|=#}
	17 &\P&\P&\P&\P&  &$\Z$  \\ \hline
	15 &  &  &  &  &  &$\Z\oplus\Z_3$\\ \hline
	13 &  &  &  &  &\F&$\Z_3$\\ \hline
	11 &  &  &\z&  &\F&      \\ \hline
	 9 &  &  &\z&  &  &      \\ \hline
	 7 &\z&  &  &  &  &      \\ \hline
	 5 &\z&  &  &  &  &      \\
	\hhline{|=======|}
\end{tabular}\endgroup

	\par\bigskip
	\begingroup
\renewcommand\arraystretch{1.2}%
\def\F{$\ \ZmodTwo\ $}%
\def\R#1{\Rnode{#1}{\F}}%
\def\P{\phantom{\F}}%
\begin{tabular}{||r||c|c|c|c|c|c||}
	\hhline{|=======|}
	$\KhModTwo$& 0& 1& 2      & 3      & 4      & 5\\ \hhline{#=#=|=|=|=|=|=#}
	17 &  &  &        &        &        &\F\\ \hline
	15 &  &  &        &        &        &\F\\ \hline
	13 &  &  &        &\R{3-13}&\R{4-13}&  \\ \hline
	11 &  &  &\R{2-11}&\R{3-11}&\R{4-11}&  \\ \hline
	 9 &  &  &\F      &        &        &  \\ \hline
	 7 &\F&  &        &        &        &  \\ \hline
	 5 &\F&\P&        &        &        &  \\
	\hhline{|=======|}
\end{tabular}%
\bockevarrow{2-11}{3-11}%
\bockoddarrow{3-11}{4-11}%
\bockoddarrow{3-13}{4-13}%
\endgroup

	\bigskip
	\caption{Even and odd Khovanov homology as well as Bockstein
	homomorphisms on the Khovanov homology over $\Z_2$ for
	the~knot $8_{19}$.}
	\label{tab:8-19}
\end{figure}

According to Example~\ref{ex:odd-vs-even-Bock} the~even and odd Bockstein homomorphisms are linearly independent. The~next two examples show the~same for $\BockOdd\BockEv$ and $\BockEv\BockOdd$.

\begin{example}
Consider the~torus knot $T_{3,4}$, labeled as $8_{19}$ in the~Rolfsen's table
\cite{Rolfsen}. Its even and odd Khovanov homology are presented in the top
two tables of Figure~\ref{tab:8-19}. An~analysis of positions of
$\ZmodTwo$ summands results in the bottom table in Figure~\ref{tab:8-19}, where
the~horizontal arrows illustrate nontrivial contributions to the~Bockstein
homomorphisms. Notice that the~composition $\BockOdd\BockEv$ does not vanish
on the~generator $u$ of $\KhModTwo^{2,11}(8_{19})\cong\Z_2$ and
$\BockMixed^2(u)\neq 0$. Therefore, the~two Bockstein homomorphisms do not
commute with each other.
\end{example}

\begin{figure}[t]
	\centering
	\begingroup
\renewcommand\arraystretch{1}%
\def\Z{$\mathbb Z$}%
\def\F{$\mathbb Z_2$}%
\def\P{}%
\begin{tabular}{||r||c|c|c|c|c|c|c|c||}
	\hhline{|=========|}
	$\EKh$& 0& 1& 2& 3& 4& 5& 6& 7\\ \hhline{#=#=|=|=|=|=|=|=|=#}
	21 &\P&\P&\P&\P&\P&\P&\P&\Z\\ \hline
	19 &  &  &  &  &  &\Z&  &\F\\ \hline
	17 &  &  &  &  &  &\Z&\Z&\P\\ \hline
  15 &  &  &  &\Z&\Z&  &  &  \\ \hline
	13 &  &  &  &\F&\Z&  &  &  \\ \hline
	11 &  &  &\Z&  &  &  &  &  \\ \hline
	 9 &\Z&  &  &  &  &  &  &  \\ \hline
	 7 &\Z&  &  &  &  &  &  &  \\
	\hhline{|=========|}
\end{tabular}\endgroup

	\qquad
	\begingroup
\renewcommand\arraystretch{1}%
\def\Z{$\mathbb Z$}%
\def\F{$\mathbb Z_2$}%
\def\G{$\mathbb Z_3$}%
\def\P{}%
\begin{tabular}{||r||c|c|c|c|c|c|c|c||}
	\hhline{|=========|}
	$\OKh$& 0& 1& 2& 3& 4& 5& 6& 7\\ \hhline{#=#=|=|=|=|=|=|=|=#}
	21 &\P&\P&\P&\P&\P&\P&  &\Z\\ \hline
	19 &  &  &  &  &  &  &\F&\Z\\ \hline
	17 &  &  &  &  &  &\G&\F&  \\ \hline
  15 &  &  &  &  &\F&\G&  &  \\ \hline
	13 &  &  &\Z&  &\F&  &  &  \\ \hline
	11 &  &  &\Z&  &  &  &  &  \\ \hline
	 9 &\Z&  &  &  &  &  &  &  \\ \hline
	 7 &\Z&  &  &  &  &  &  &  \\
	\hhline{|=========|}
\end{tabular}\endgroup

	\par\bigskip
	\begingroup
\renewcommand\arraystretch{1.2}%
\def\F{$\ \ZmodTwo\ $}%
\def\R#1{\Rnode{#1}{\F}}%
\def\P{\phantom{\F}}%
\begin{tabular}{||r||c|c|c|c|c|c|c|c||}
	\hhline{|=========|}
	 $\KhModTwo$& 0& 1& 2      & 3      & 4      & 5      & 6      & 7      \\ \hhline{#=#=|=|=|=|=|=|=|=#}
	21 &\P&\P&\P      &\P      &\P      &\P      &\P      &\F      \\ \hline
	19 &  &  &        &        &        &\R{5-19}&\R{6-19}&\R{7-19}\\ \hline
	17 &  &  &        &        &        &\R{5-17}&\R{6-17}&\P      \\ \hline
  15 &  &  &        &\R{3-15}&\R{4-15}&        &        &        \\ \hline
	13 &  &  &\R{2-13}&\R{3-13}&\R{4-13}&        &        &        \\ \hline
	11 &  &  &\F      &        &        &        &        &        \\ \hline
	 9 &\F&  &        &        &        &        &        &        \\ \hline
	 7 &\F&  &        &        &        &        &        &        \\
	\hhline{|=========|}
\end{tabular}%
\bockevarrow{2-13}{3-13}
\bockoddarrow{3-13}{4-13}%
\bockoddarrow{3-15}{4-15}%
\bockoddarrow{5-19}{6-19}%
\bockoddarrow{5-17}{6-17}%
\bockevarrow{6-19}{7-19}%
\endgroup

	\bigskip
	\caption{Even and odd Khovanov homology as well as Bockstein
	homomorphisms on the Khovanov homology over $\Z_2$ for
	the~knot $10_{124}$.}
	\label{tab:10-124}
\end{figure}

\begin{example}
The~torus knot $T_{3,5}$, labeled as $10_{124}$ in the~Rolfsen's
table~\cite{Rolfsen}, admits a~class with the~opposite property: $\BockEv\BockOdd$ does not not vanish on the~generator
$u$ of $\KhModTwo^{5,19}(10_{124})\cong\Z_2$, see Figure~\ref{tab:10-124}.
\end{example}

The~even Khovanov homology is multiplicative with respect to disjoint union and connected sum of links:
\begin{gather}
	\EKhCom(D\sqcup D') \cong \EKhCom(D) \tensor[\phantom{{}_2}\ZmodTwo] \EKhCom(D'),\\
	\EKhCom(D\# D') \cong \EKhCom(D) \tensor[A] \EKhCom(D'),
\end{gather}
where $A=\Z v_+\oplus\Z v_-$ is the~Frobenius algebra associated to a~circle
\cite{KhHom,KhPatterns}. In the~latter formula, we regard $\EKhCom(D)$ as
an~$A$--module with the~action of $A$ induced by the~cobordism that merges
a~circle to the~link diagram at the~place where the~connected sum is
performed, see Figure~\ref{fig:circle-merge}.

\begin{figure}[ht]
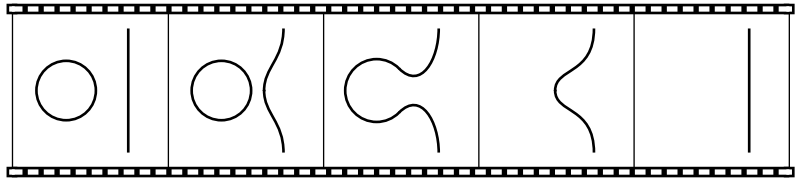

	\psset{linewidth=1pt,unit=0.0625\textwidth}
	\begin{movie}[hb](2,2){5}
		\movieclip{%
			\pscircle(-0.3,0){0.4}
			\psline(0.5,-0.8)(0.5,0.8)
		}
		\movieclip{%
			\pscircle(-0.3,0){0.4}
			\psbezier{-c}(0.5,-0.8)(0.5,-0.4)(0.25,-0.3)(0.25,0)
			\psbezier{-c}(0.5, 0.8)(0.5, 0.4)(0.25, 0.3)(0.25,0)
		}
		\movieclip{%
			\psarc(-0.3,0){0.405}{45}{315}
			\psbezier{-c}(0.5,-0.8)(0.5,-0.4)(0.27268, 0)(-0.01016,-0.28284)
			\psbezier{-c}(0.5, 0.8)(0.5, 0.4)(0.27268, 0)(-0.01016, 0.28284)
		}
		\movieclip{%
			\psbezier{-c}(0.5,-0.8)(0.5,-0.2)(0,-0.3)(0,0)
			\psbezier{-c}(0.5, 0.8)(0.5, 0.2)(0, 0.3)(0,0)
		}
		\movieclip{%
			\psline(0.5,-0.8)(0.5,0.8)
		}
	\end{movie}
	\caption{Sections of the~cobordism inducing an~action of $A$ on the~Khovanov chain complex $\EKhCom(D)$.}
	\label{fig:circle-merge}
\end{figure}

The~same formulas were recently proved for the~odd Khovanov chain
complexes~\cite{KhHomTensor}.  Since $\ZmodTwo$ is a~field, the~K\"unneth
formula identifies the~homology of the~disjoint union or connected sum as
tensor products of homology:
\begin{gather}
	\label{eq:disjsum-vs-tensor}
		\KhModTwo(L\sqcup L') \cong
			\KhModTwo(L) \tensor[\phantom{{}_2}\ZmodTwo]  \KhModTwo(L'),\\
	\label{eq:connsum-vs-tensor}
		\KhModTwo(L\# L') \cong \KhModTwo(L) \tensor[A] \KhModTwo(L').
\end{gather}
\begin{definition}
A~homological operation $\theta_L\colon \KhModTwo^i(L) \to \KhModTwo^{i+d}(L)$
of degree $d$ is said to be
\begin{itemize}
	\item \emph{$\sqcup$--primitive} if $\theta_{L\sqcup L'} =
						\theta_L \tensor[\phantom{{}_2}\ZmodTwo] \id +
						\id      \tensor[\phantom{{}_2}\ZmodTwo] \theta_{L'}$ and
	\item \emph{$\#$--primitive} if $\theta_{L\# L} =
						\theta_L \tensor[A] \id +
						\id      \tensor[A] \theta_{L'}$
\end{itemize}
for every two links $L$ and $L'$.
\end{definition}

\begin{corollary}
	All three Bockstein operations are $\sqcup$--primitive and $\#$--primitive.
\end{corollary}

This allows us to construct knots with nontrivial compositions of four Bockstein homomorphisms.

\begin{example}\label{ex:4-bocks}
	Consider the~generators $u \in \KhModTwo^{2,13}(10_{124})$ and $u'\in\KhModTwo^{5,19}(10_{124})$. Then $\BockOdd(u)=0$, $\BockEv(u')=0$, $\BockOdd\BockEv(u)\not=0$, and $\BockEv\BockOdd(u')\not=0$ (see Fig.~\ref{tab:10-124}). It is now straightforward to verify that $u\otimes u'\in \KhModTwo^{7,32}(10_{124}\# 10_{124})$ satisfies
	\begin{equation}
		\BockOdd\BockEv\BockOdd\BockEv(u\otimes u') =
		\BockEv\BockOdd\BockEv\BockOdd(u\otimes u') =
		\BockOdd\BockEv(u)\otimes\BockEv\BockOdd(u') \neq 0.
	\end{equation}
	In particular, both compositions are nontrivial, but $\BockMixed^4(u\otimes u') = 0$.
\end{example}

The~vanishing of $\BockMixed^4$ above is not surprising---it is a~well-known fact that given a~primitive operation $\theta$, each power $\theta^{2^r}$ is also primitive. Indeed, one first computes
\begin{equation}\label{eq:binomial-for-primitive}
	\theta^n(x\otimes y) = \sum_{k=0}^n {n\choose k} \theta^{n-k}(x) \otimes \theta^k(y)
\end{equation}
and then checks that ${2^r\choose k}$ is even for all values of $k$ except $k=0$ and  $k=2^r$. In particular, $\BockMixed^{2^r}$ vanishes on the~homology of $K_1\# K_2$ if it vanishes on homology of both $K_1$ and $K_2$. A~quick look at Figure~\ref{tab:10-124} reveals that already $\BockMixed^3=0$ for the knot~$10_{124}$.

According to the~discussion after Proposition~\ref{prop:three-operations}, the~existence of nontrivial compositions of length bigger than $\ell$ implies that $\BockMixed^{\ell}\neq 0$. Hence, it is important to find knots for which higher powers of $\BockMixed$ do not vanish.

The~torus knot $T_{5,6}$ admits a~class $u\in \KhModTwo^{5,31}(T_{5,6})$, such that $\BockEv(u)=0$ but $\BockOdd\BockEv\BockOdd(u)\neq 0$. Unfortunately, larger knots are beyond the~capabilities of our computer program~\cite{KhoHo}, so that we did not succeed in finding a~knot for which $\BockMixed^4 \neq 0$. Nonetheless, we expect that there are examples of knots for which higher powers of $\BockMixed$ are nonzero.

\begin{conjecture}\label{conj:Bocksteins-algebra}
	The~operation $\BockMixed^n$ does not vanish for the~torus knot $T_{2n-1,2n}$. In particular, for every $n\in\mathbb N$ there exists a~knot $K_n$ such that the~operations $\BockMixed^n$, $\BockMixed^{n-1}\BockEv$, and $\BockMixed^{n-1}\BockOdd$ on $\KhModTwo(K_n)$ are nonzero and pairwise different.
\end{conjecture}

\section{Integral lifts}\label{sec:oper-integral}

As explained in Section~\ref{sec:pullback},
$\UKh(L)$ can be regarded as an~extension of both even and odd Khovanov
homology theories. This leads to homological operations between integral
Khovanov homologies of a~link:
\begin{enumerate}
	
	\item\label{def:ev-odd-bock}
	$\BockEvOdd\colon\EKh^i(L)\to\OKh^{i+1}(L)$, the~connecting
	homomorphism for the~sequence \eqref{ses:Kh-odd-pi-even} and
	
	\item\label{def:odd-ev-bock}
	$\BockOddEv\colon\OKh^i(L)\to\EKh^{i+1}(L)$, the~connecting
	homomorphism for the~sequence \eqref{ses:Kh-even-pi-odd}.
	
\end{enumerate}
Both operations are link invariants, because they arise as Bockstein homomorphisms associated to the~short exact sequences of $\Zunfd$--modules
\begin{gather}
\label{ses:Z-odd-pi-even}
	0\to\Zodd\to\Zunfd\to\Zev\to 0\mathrlap{,\text{ and}}\\
\label{ses:Z-even-pi-odd}
	0\to\Zev\to\Zunfd\to\Zev\to 0
\end{gather}
respectively, considered as coefficients for the~unified Khovanov homology $\UKhCom(D)$. As in the~case of Khovanov complexes, these sequences are consequences of the~pullback description of $\Zunfd$.

Tensoring the~exact sequences \eqref{ses:Z-odd-pi-even} and
\eqref{ses:Z-even-pi-odd} with $\ZmodTwo$ reveals that $\BockEvOdd$ and
$\BockOddEv$ are integral lifts of Bockstein homomorphisms from the previous
section.

\begin{proposition}\label{prop:lifts-of-bocks}
	Given a~link $L$, there are commuting squares
	\begin{equation}\label{diag:lifts-of-bocks}
		\begin{diagps}(0em,-0.8ex)(24em,12.25ex)
			\square<8em,10ex>[%
				\EKh^i(L)`\OKh^{i+1}(L)`\KhModTwo^i(L)`\KhModTwo^{i+1}(L);
				\BockEvOdd```\BockOdd
			]
			\square(16em,0ex)<8em,10ex>[%
				\OKh^i(L)`\EKh^{i+1}(L)`\KhModTwo^i(L)`\KhModTwo^{i+1}(L);
				\BockOddEv```\BockEv
			]
		\end{diagps}
	\end{equation}
\end{proposition}
\begin{proof}
We start with finding a~formula for $\BockEvOdd$ using the~pullback
description of $\UKhCom(L)$. Pick a~cocycle $x\in\EKhCom(L)$; it is covered by
$(x,x)\in\UKhCom(L)$, where we identify the~even and odd chain groups in
the~natural way. Then $\diffUnfd(x,x) = (0,\diffOdd x)$ is the~image of
$\tfrac{1}{2}\diffOdd x$, as the~inclusion of the~odd homology takes a~chain
$y$ into $(0,2y)$. Notice that the~division makes sense because modulo~$2$
we have $\diffOdd x\equiv\diffEv x = 0$. Hence, $\BockEvOdd[x] =
\left[\frac{1}{2}\diffOdd x\right]$, which agrees modulo 2 with
$\BockOdd([x]\tensor\ZmodTwo)$. The~case of $\BockOddEv$ is proved likewise.
\end{proof}

\begin{corollary}\label{cor:int-Bocks-vs-2}
	Both $\BockEvOdd$ and $\BockOddEv$ are annihilated by 2.
\end{corollary}
\begin{proof}
Since $\BockEvOdd[x] = \left[\frac{1}{2}\diffOdd x\right]\in\OKh(L)$, we
immediately get that $\BockEvOdd[2x]=[\diffOdd x]=0$. The case of $\BockOddEv$
is similar.
\end{proof}

Combining Proposition~\ref{prop:lifts-of-bocks} and Conjecture~\ref{conj:Bocksteins-algebra} results in an~algebraic independence of integral operations.

\begin{conjecture}\label{conj:int-Bockstein-algebra}
	All alternating compositions $\cdots\BockEvOdd\BockOddEv$ and $\cdots\BockOddEv\BockEvOdd$ are nontrivial.
\end{conjecture}

Indeed, it is enough to find a~link $L$ and a~class $a\in\OKh(L)$ (resp.\ $a\in\EKh(L)$) such that the~composition $\cdots\BockOdd\BockEv$ (resp.\ $\cdots\BockEv\BockOdd$) does not vanish on the~$\ZmodTwo$--reduction $\bar a\in\KhModTwo(L)$. In particular, we know that
\begin{itemize}
	\item both $\BockEvOdd\BockOddEv$ and $\BockOddEv\BockEvOdd$ are nonzero for $10_{124} = T_{3,5}$, and
	\item $\BockEvOdd\BockOddEv\BockEvOdd$ is nonzero for $T_{5,6}$.
\end{itemize}
We expect higher torus knots to provide examples for which longer compositions are nontrivial.

The~ring $\Zunfd$ is the~only nontrivial extension between $\Zev$ and $\Zodd$. Indeed, free resolutions of $\Zev$ and $\Zodd$ are given by infinite sequences
\begin{align}
	0\to/<-/ \Zev &
		 \to/<-/ \Zunfd
		 \to/<-/^{1-\oddParameter} \Zunfd
		 \to/<-/^{1+\oddParameter} \Zunfd
		 \to/<-/ \ldots \\
	0\to/<-/ \Zodd &
		 \to/<-/ \Zunfd
		 \to/<-/^{1+\oddParameter} \Zunfd
		 \to/<-/^{1-\oddParameter} \Zunfd
		 \to/<-/ \ldots
\end{align}
from which one computes
\begin{align}
		\Ext^1_{\Zunfd}(\Zev,\Zev)  &\cong 0, &	\Ext^1_{\Zunfd}(\Zev,\Zodd)&\cong\ZmodTwo,\\
		\Ext^1_{\Zunfd}(\Zodd,\Zodd)&\cong 0,	&	\Ext^1_{\Zunfd}(\Zodd,\Zev)&\cong\ZmodTwo.
\end{align}

\begin{corollary}
	The~operations $\BockEvOdd$ and $\BockOddEv$ are the~only nontrivial Bockstein-type operations between the~even and odd Khovanov homology.
\end{corollary}

Similarly to the~$\ZmodTwo$ case, the~operations $\BockEvOdd$ and $\BockOddEv$ are primitive when regarded as operations acting on the~direct sum of the~even and odd Khovanov homology. Indeed, if we write $\BKh(L):=\EKh(L)\oplus\OKh(L)$, then we can identify $\BockEvOdd$ and $\BockOddEv$ with matrices
\begin{equation}
	\BockEvOddDS := \begin{pmatrix} 0 & 0 \\ \BockEvOdd & 0 \end{pmatrix}
	\hskip 2cm
	\BockOddEvDS := \begin{pmatrix} 0 & \BockOddEv \\ 0 & 0 \end{pmatrix}
\end{equation}
respectively. Following the~K\"unneth formula, we define for links $L$ and
$L'$
\begin{equation}
	\EKh(L) \dertensor \EKh(L') :=
			\EKh(L) \tensor \EKh(L')  \oplus  \Tor(\EKh(L),\EKh(L'))[1]
			\cong \EKh(L\sqcup L')
\end{equation}
and likewise for $\OKh$ and $\BKh$. Then $\BKh(L\sqcup L') \cong \EKh(L) \dertensor \EKh(L') \oplus \OKh(L) \dertensor \OKh(L')$ is a~direct summand of $\BKh(L) \dertensor \BKh(L')$, and the~operations $\BockEvOdd$ and $\BockOddEv$ regarded as endomorphisms of the~product $\BKh(L) \dertensor \BKh(L')$ can be identified with matrices
\begin{gather}
	\BockEvOddDS ={
		\begin{pmatrix}
			0 & 0 & 0 & 0 \\
			\id\dertensor\BockEvOdd & 0 & 0 & 0 \\
			\BockEvOdd\dertensor\id & 0 & 0 & 0 \\
			0 & \BockEvOdd\dertensor\id & \id\dertensor\BockEvOdd & 0
		\end{pmatrix}}
	= \BockEvOddDS\dertensor\id_{\BKh} + \id_{\BKh}\dertensor\BockEvOddDS
	,\textrm{ and}
	\\[1ex]
	\BockOddEvDS ={
		\begin{pmatrix}
			0 & \id\dertensor\BockOddEv & \BockOddEv\dertensor\id & 0 \\
			0 & 0 & 0 & \BockOddEv\dertensor\id \\
			0 & 0 & 0 & \id\dertensor\BockOddEv \\
			0 & 0 & 0 & 0
		\end{pmatrix}}
	= \BockOddEvDS\dertensor\id_{\BKh} + \id_{\BKh}\dertensor\BockOddEvDS
	,\phantom{\textrm{ and}}
\end{gather}
where we order the~four summands of $\BKh(L)\dertensor\BKh(L')$ as
\emph{even}--\emph{even}, \emph{even}--\emph{odd}, \emph{odd}--\emph{even},
and \emph{odd}--\emph{odd}. In particular, $\Phi := \BockEvOddDS +
\BockOddEvDS = \begin{psmallmatrix} 0 & \BockOddEv \\ \BockEvOdd & 0
\end{psmallmatrix}$ is a~$\sqcup$--primitive operation on $\BKh(L)$. The~same
argument works for connected sums of links, showing the~operations are
$\#$--primitive as well.

\begin{remark}
Using the~primitivity as defined above, one may try to prove Conjecture~\ref{conj:int-Bockstein-algebra} without referring to the~$\ZmodTwo$--operations. Unfortunately, $\Phi^{2^r}$ is again primitive for every $r>0$, due to the~binomial formula \eqref{eq:binomial-for-primitive} and Corollary~\ref{cor:int-Bocks-vs-2}.
\end{remark}

\section{Reduced Khovanov homology}\label{sec:reduced}

Given a~diagram $D$ of a~knot $K$, choose a~point $b$ on $D$ and consider
the~cobordism merging a~circle to $D$ at the~point $b$, as shown in
Figure~\ref{fig:circle-merge}. This operation induces
on the~chain complex $\UKhCom(D)$ a~module structure over
$A_\oddParameter=\Zunfd v_+\oplus\Zunfd v_-$, the~algebra associated to
a~circle. This structure is independent of the~chosen point $b$, see \cite{KhHomTensor}.

Consider the~subcomplex $\RedUKhCom(D)$ of $\UKhCom(D)$ spanned by the~chains
annihilated by $v_-$. We call its homology the~\emph{reduced unified homology}
$\RedUKh(K)$ of $K$.  A~quick comparison with definitions of the reduced even
and odd Khovanov homology \cite{KhPatterns,OddKhHom} reveals that both are
specializations of $\RedUKh(K)$, cf.~\eqref{eq:EKh-and-OKh-from-UKh}:
\begin{equation}
	\RedEKh(K) \cong \RedUKh(K;\Zev),
	\hskip 2cm
	\RedOKh(K) \cong \RedUKh(K;\Zodd).
\end{equation}
In particular, there are reduced versions of all the~homological operations
defined so far:
\begin{itemize}
	\item the~reduced Bockstein operations $\RedBockEv,\RedBockOdd\colon \RedKhModTwo^i(K)\to \RedKhModTwo^{i+1}(K)$, and
	\item their integral lifts $\RedBockEvOdd\colon \RedEKh^i(K)\to \RedOKh^{i+1}(K)$ and $\RedBockOddEv\colon \RedOKh^i(K)\to \RedEKh^{i+1}(K)$.
\end{itemize}
Clearly, the~reduced homology is multiplicative with respect to the~connected
sum of knots, which implies that reductions of $\#$--primitive operations are
primitive. In particular, all the~reduced Bockstein homomorphisms are
primitive. As before, we conjecture that the~reduced Bockstein homomorphisms satisfy only the~obvious relations.

\begin{conjecture}
	The~alternating compositions of reduced Bockstein homomorphisms are nonzero and different, and so are their integral lifts (given as alternating compositions of $\RedBockEvOdd$ and $\RedBockOddEv$).
\end{conjecture}

\begin{example}
	The~torus knot $T_{4,5}$ admits a~class $u\in\RedKhModTwo(T_{4,5})$, on which $\RedBockEv$ vanishes but
$\RedBockEv\RedBockOdd$ does not, see Figure~\ref{tab:T(4,5)}.
\end{example}

\begin{figure}[tp]
	\centering
	\begingroup
\small
\renewcommand\arraystretch{1}%
\def\Z{$\mathbb Z$}%
\def\F{$\mathbb Z_2$}%
\def\G{$\mathbb Z_3$}%
\begin{tabular}{||r||c|c|c|c|c|c|c|c|c|c|c||}
	\hhline{|============|}
	$\RedEKh$& 0& 1& 2& 3& 4& 5& 6& 7& 8& 9&10\\ \hhline{#=#=|=|=|=|=|=|=|=|=|=|=#}
	28 &  &  &  &  &  &  &  &  &  &  &\F\\ \hline
	26 &  &  &  &  &  &  &  &  &  &\Z&  \\ \hline
	24 &  &  &  &  &  &  &  &\Z&\Z&  &  \\ \hline
	22 &  &  &  &  &  &\Z&  &\F&  &  &  \\ \hline
  20 &  &  &  &  &  &  &\Z&  &  &  &  \\ \hline
	18 &  &  &  &\Z&\Z&  &  &  &  &  &  \\ \hline
	16 &  &  &\Z&  &  &  &  &  &  &  &  \\ \hline
	14 &  &  &  &  &  &  &  &  &  &  &  \\ \hline
	12 &\Z&  &  &  &  &  &  &  &  &  &  \\
	\hhline{|============|}
\end{tabular}\endgroup

	\par\bigskip
	\begingroup
\small
\renewcommand\arraystretch{1}%
\def\Z{$\mathbb Z$}%
\def\F{$\mathbb Z_2$}%
\def\G{$\mathbb Z_3$}%
\begin{tabular}{||r||c|c|c|c|c|c|c|c|c|c|c||}
	\hhline{|============|}
	$\RedOKh$& 0& 1& 2& 3& 4& 5& 6& 7& 8& 9&10\\ \hhline{#=#=|=|=|=|=|=|=|=|=|=|=#}
	28 &  &  &  &  &  &  &  &  &  &\Z&\Z\\ \hline
	26 &  &  &  &  &  &  &  &  &  &\Z&  \\ \hline
	24 &  &  &  &  &  &  &  &\Z&\Z&  &  \\ \hline
	22 &  &  &  &  &  &  &\F&\Z&  &  &  \\ \hline
  20 &  &  &  &  &  &\G&\Z&  &  &  &  \\ \hline
	18 &  &  &  &  &\F&  &  &  &  &  &  \\ \hline
	16 &  &  &\Z&  &  &  &  &  &  &  &  \\ \hline
	14 &  &  &  &  &  &  &  &  &  &  &  \\ \hline
	12 &\Z&  &  &  &  &  &  &  &  &  &  \\
	\hhline{|============|}
\end{tabular}\endgroup

	\par\bigskip
	\begingroup
\small
\renewcommand\arraystretch{1.2}%
\def\F{$\ \ZmodTwo\ $}%
\def\P{\phantom{\F}}%
\def\R#1{\Rnode{#1}{\F}}%
\begin{tabular}{||r||c|c|c|c|c|c|c|c|c|c|c||}
	\hhline{|============|}
	$\RedKhModTwo$& 0& 1& 2& 3      & 4      & 5      & 6      & 7      & 8& 9      &10       \\ \hhline{#=#=|=|=|=|=|=|=|=|=|=|=#}
	28 &  &\P&  &        &        &        &        &        &  &\R{9-28}&\R{10-28}\\ \hline
	26 &  &  &  &        &        &        &        &        &  &\F      &         \\ \hline
	24 &  &  &  &        &        &        &        &\F      &\F&        &         \\ \hline
	22 &  &  &  &        &        &\R{5-22}&\R{6-22}&\R{7-22}&  &        &         \\ \hline
  20 &  &  &  &        &        &        &\F      &        &  &        &         \\ \hline
	18 &  &  &  &\R{3-18}&\R{4-18}&        &        &        &  &        &         \\ \hline
	16 &  &  &\F&        &        &        &        &        &  &        &         \\ \hline
	14 &  &  &  &        &        &        &        &        &  &        &         \\ \hline
	12 &\F&  &  &        &        &        &        &        &  &        &         \\
	\hhline{|============|}
\end{tabular}%
\bockoddarrow{3-18}{4-18}%
\bockevarrow{9-28}{10-28}%
\bockoddarrow{5-22}{6-22}%
\bockevarrow{6-22}{7-22}%
\endgroup

	\vskip 0.5\baselineskip
	\caption{Even and odd reduced Khovanov homology as well as Bockstein
	homomorphisms on the reduced Khovanov homology over $\Z_2$ for
	the~knot $T_{4,5}$.}
	\label{tab:T(4,5)}
\end{figure}

\section{Experimental results}\label{sec:computation}

We computed ranks of all the homological operations discussed in this paper in every bigrading for all prime knots with up to 16 crossings using~{\tt KhoHo}~\cite{KhoHo}. We then compared these ranks for pairs of knots that have the same even and odd Khovanov homology. As discussed in the~Introduction, the~ranks of Bockstein homomorphisms $\BockEv$ and $\BockOdd$ are fully determined by the~even and odd Khovanov homology. Therefore, they do not provide any new information. However, the~mixed Bockstein $\BockMixed=\BockEv+\BockOdd$ is different. It turns out that there exist~$7$ pairs of prime knots ($14$ if counted with mirror images) with up to $14$ crossings that have the same even and odd Khovanov homology but different ranks of $\BockMixed$. There are $61$ (resp.~$122$) such pairs among knots with up to $15$ crossings and $742$ (resp.~$1484$) with up to $16$ crossings, see Table~\ref{tab:distinguished-pairs}.

\begin{corollary}
	The~unified homology $\UKh(L)$ is a~stronger link invariant than $\EKh(L)\oplus\OKh(L)$.
\end{corollary}

\begin{table}[t]
\begin{tabular}{l|c|c|c}
& $\le14$ crossings & $\le15$ crossings & $\le16$ crossings \\
\hline
Number of knots (counting mirrors)       &67289&403348&2420670\\
Pairs distinguished by ranks of $\BockMixed$   &  14 &  122 &  1484 \\
Pairs distinguished by ranks of $\BockMixed^2$ &   0 &    8 &   108 \\
Pairs distinguished by ranks of $\BockOddEv$   &   7 &   61 &   760 \\
Pairs distinguished by ranks of $\BockEvOdd$   &   7 &   61 &   760 \\
Pairs distinguished by ranks of $\BockSqEv$    &   0 &    4 &    71 \\
Pairs distinguished by ranks of $\BockSqOdd$   &   9 &   95 &  1232 \\
\hline
\end{tabular}
\bigskip
\caption{Number of pairs of prime non-alternating knots with the same even and odd Khovanov homology that are distinguished by homological operations}
\label{tab:distinguished-pairs}
\end{table}

Here are the first $7$ pairs of knots\footnote{
	Here, $13^n_{141}$ denotes the non-alternating knot number 141 with 13 crossings from the Knotscape knot table and $\overline{14}^n_{21021}$ is the mirror image of the knot $14^n_{21021}$.}
distinguished by $\BockMixed$:
\begin{equation}\label{eq:beta-pairs}
\begin{aligned}
  13^n_{141}    &\to/<->/<1em>  14^n_{2551} &\quad
  13^n_{1002}   &\to/<->/<1em>  14^n_{6487} &\quad
  14^n_{1346}   &\to/<->/<1em>  14^n_{7711} &\quad
  14^n_{5293}   &\to/<->/<1em>  14^n_{12516} \\
  14^n_{5373}   &\to/<->/<1em>  14^n_{12516} &
  14^n_{6632}   &\to/<->/<1em>  \overline{14}^n_{21021} &
  14^n_{12393}   &\to/<->/<1em>  14^n_{12532}
\end{aligned}
\end{equation}

Existence of such pairs can be explained by the~observation that Bockstein
homomorphisms are described by a~noncanonical splitting $H^i(C,\ZmodTwo)\cong
H^i(C)\tensor\ZmodTwo \oplus \Tor(H^{i+1}(C),\ZmodTwo)$, and in case of
Khovanov homology the~two splittings, one for the~even and one for the~odd
version, do not coincide. In other words, we cannot pick such isomorphisms for
even and odd Khovanov homology that agree over $\ZmodTwo$.

$\BockMixed^2$ is, obviously, a much weaker invariant than $\BockMixed$ since
$\rk(\BockMixed^2)\le\rk\BockMixed$ (in an appropriate bigrading). 
Nonetheless, there are $4$ (resp.~$8$) pairs of knots with $15$ crossings that
can be distinguished by $\BockMixed^2$:
\begin{equation}
\begin{aligned}
  15^n_{23106}  &\to/<->/<1em>  15^n_{56014} &\qquad
  15^n_{23432}  &\to/<->/<1em>  15^n_{56014} \\
  15^n_{44028}  &\to/<->/<1em> \overline{15}^n_{50224} &\qquad
  15^n_{73047}  &\to/<->/<1em>  \overline{15}^n_{91280}
\end{aligned}
\end{equation}

It is important to notice that the Khovanov homology modulo $2$ for a knot $K$
and its mirror image $\overline{K}$ are dual to each other. As such, if any of
the homological operations over $\Z_2$ has the same ranks for two knots, say,
$K_1$ and $K_2$, then the ranks are the same for $\overline{K}_1$ and
$\overline{K}_2$ as well. The situation is different for integral operations.
It turns out that among all knots with up to $14$ crossings, both $\BockOddEv$
and $\BockEvOdd$ distinguish $7$ pairs of knots, but not their mirror images,
cf.~\eqref{eq:beta-pairs}:
\begin{equation}\label{eq:phi-pairs}
\begin{aligned}
  \overline{13}^n_{141}    &\to/<->/<1em>  \overline{14}^n_{2551} &\quad
  13^n_{1002}   &\to/<->/<1em>  14^n_{6487} &\quad
  14^n_{1346}   &\to/<->/<1em>  14^n_{7711} &\quad
  14^n_{5293}   &\to/<->/<1em>  14^n_{12516} \\
  14^n_{5373}   &\to/<->/<1em>  14^n_{12516} &
  14^n_{6632}   &\to/<->/<1em>  \overline{14}^n_{21021} &
  14^n_{12393}   &\to/<->/<1em>  14^n_{12532}
\end{aligned}
\end{equation}

\phantomsection\label{def:ev-sq-bock}%
\phantomsection\label{def:odd-sq-bock}%
One can consider integral homological operations of degree $2$ as well: $\BockSqEv := \BockOddEv\BockEvOdd$ and $\BockSqOdd := \BockEvOdd\BockOddEv$, the~former defined for the~even and the~latter for the~odd Khovanov homology. Both of them are, obviously, integral lifts of $\BockMixed^2$, see~\eqref{diag:lifts-of-bocks}. Computations reveal that $\BockSqOdd$ distinguishes the same $7$ pairs of knots with at most $14$ crossings as in~\eqref{eq:phi-pairs}, plus two more:
\begin{equation}\label{eq:thetaO-pairs}
	13^n_{651}   \to/<->/<1em>  14^n_{16550}\qquad
	13^n_{661}   \to/<->/<1em>  14^n_{16550}
\end{equation}
On the other hand, the first pair of knots with the same even and odd Khovanov homology but with different ranks of $\BockSqEv$ has 15 crossings, see Table~\ref{tab:distinguished-pairs}. There are $4$ such pairs in total:
\begin{equation}
\begin{aligned}
	\overline{15}^n_{23106} &\to/<->/<1em>  \overline{15}^n_{56014}  &\qquad
	\overline{15}^n_{23432} &\to/<->/<1em>  \overline{15}^n_{56014}  \\
	\overline{15}^n_{44028} &\to/<->/<1em>  15^n_{50224}  &\qquad
	\overline{15}^n_{73047} &\to/<->/<1em>  15^n_{91280}
\end{aligned}
\end{equation}

\begin{figure}[p]%
	\centering
	\begingroup
\renewcommand\arraystretch{1}%
\def\F{\Z_2}%
\def\o{{\oplus}}
\begin{tabular}{||r||c|c|c|c|c|c|c|c|c|c|c||}
	\hhline{|============|}
	$\OKh$&$-6$&$ -5 $&$ -4 $&$ -3 $&$ -2 $&$ -1 $&$  0 $&$  1 $&$  2 $&$  3 $&$ 4$\\ \hhline{#=#=|=|=|=|=|=|=|=|=|=|=#}
	$11 $&    &      &      &      &      &      &      &      &      &      &$\Z$\\ \hline
	$ 9 $&    &      &      &      &      &      &      &      &      &$\Z^2$&$\Z$\\ \hline
	$ 7 $&    &      &      &      &      &      &      &      &$\Z^4$&$\Z^2$&    \\ \hline
  $ 5 $&    &      &      &      &      &      &      &$\Z^5$&$\Z^4\o\F$&  &    \\ \hline
	$ 3 $&    &      &      &      &      &      &\Rnode{dom3}{$\Z^6$}&$\Z^5\o\F$&\Rnode{cod3}{$\F$}&    &    \\ \hline
	$ 1 $&    &      &      &      &      &\Rnode{dom1}{$\Z^7$}&$\Z^6\o\F$&\Rnode{cod1}{$\F$}&    &      &    \\ \hline
	$-1 $&    &      &      &      &\Rnode{dom-1}{$\Z^6$}&$\Z^7\o\F$&\Rnode{cod-1}{$\F$}&    &      &      &    \\ \hline
	$-3 $&    &      &      &\Rnode{dom-3}{$\Z^5$}&$\Z^6\o\F$&\Rnode{cod-3}{$\F$}&    &      &      &      &    \\ \hline
	$-5 $&    &      &\Rnode{dom-5}{$\Z^4$}&$\Z^5$&\Rnode{cod-5}{$\F$}&      &      &      &      &      &    \\ \hline
	$-7 $&    &$\Z^2$&$\Z^4$&      &      &      &      &      &      &      &    \\ \hline
	$-9 $&$\Z$&$\Z^2$&      &      &      &      &      &      &      &      &    \\ \hline
	$-11$&$\Z$&      &      &      &      &      &      &      &      &      &    \\
	\hhline{|============|}
\end{tabular}%
\diagarc[border=2\pslinewidth,linecolor=blue,arcangle=20]{->}{dom-5}{cod-5}%
\diagarc[border=2\pslinewidth,linecolor=blue,arcangle=20]{->}{dom-3}{cod-3}%
\diagarc[border=2\pslinewidth,linecolor=blue,arcangle=20]{->}{dom-1}{cod-1}%
\diagarc[border=2\pslinewidth,linecolor=blue,arcangle=20]{->}{dom1}{cod1}%
\diagarc[border=2\pslinewidth,linecolor=red,linestyle=dashed,arcangle=20]{->}{dom3}{cod3}%
\endgroup

	\par\bigskip
	\begingroup
\renewcommand\arraystretch{1}%
\def\F{\Z_2}%
\def\o{{\oplus}}
\begin{tabular}{||r||c|c|c|c|c|c|c|c|c|c|c||}
	\hhline{|============|}
	$\OKh$&$-4$&$ -3 $&$ -2 $&$ -1 $&$  0 $&$  1 $&$  2 $&$  3 $&$  4 $&$  5 $&$ 6$\\ \hhline{#=#=|=|=|=|=|=|=|=|=|=|=#}
	$11 $&    &      &      &      &      &      &      &      &      &      &$\Z$\\ \hline
	$ 9 $&    &      &      &      &      &      &      &      &      &$\Z^2$&$\Z$\\ \hline
	$ 7 $&    &      &      &      &      &      &      &      &$\Z^4$&$\Z^2$&    \\ \hline
  $ 5 $&    &      &      &      &      &      &      &$\Z^5\o\F$&$\Z^4$&  &    \\ \hline
	$ 3 $&    &      &      &      &      &      &$\Z^6\o\F$&$\Z^2\o\F$&&    &    \\ \hline
	$ 1 $&    &      &      &      &      &$\Z^7\o\F$&$\Z^6\o\F$&&    &      &    \\ \hline
	$-1 $&    &      &      &      &$\Z^6\o\F$&$\Z^7\o\F$&&    &      &      &    \\ \hline
	$-3 $&    &      &      &$\Z^5\o\F$&$\Z^6\o\F$&&    &      &      &      &    \\ \hline
	$-5 $&    &      &$\Z^4$&$\Z^5\o\F$&  &      &      &      &      &      &    \\ \hline
	$-7 $&    &$\Z^2$&$\Z^4$&      &      &      &      &      &      &      &    \\ \hline
	$-9 $&$\Z$&$\Z^2$&      &      &      &      &      &      &      &      &    \\ \hline
	$-11$&$\Z$&      &      &      &      &      &      &      &      &      &    \\
	\hhline{|============|}
\end{tabular}\endgroup

	\caption[Odd Khovanov homology for knots $13^n_{1002}$ and
	$14^n_{6487}$]{Odd Khovanov homology for knots $13^n_{1002}$ and
	$14^n_{6487}$ (upper table) and their mirror images (lower table).
	The~solid blue arrows indicate places where the~operation
	$\BockSqOdd$ is surjective for both knots. At the~dashed red
	arrow, $\BockSqOdd$ is surjective for $13^n_{1002}$, but is a~zero map
	for $14^n_{6487}$. Notice that there is no place for a~nontrivial
	$\BockSqOdd$ in the~lower table.}
	\label{tab:13-14}
\end{figure}

\begin{figure}[p]%
	\centering
	\setlength\tabcolsep{3pt}
	\begingroup
\Small
\renewcommand\arraystretch{1.1}%
\def\F{\Z_2}%
\def\G{\Z_3}%
\def\o{{\oplus}}
\def\s#1{\Small$#1$}
\def\t#1{\tiny$#1$}
\begin{tabular}{||r||c|c|c|c|c|c|c|c|c|c|c|c|c|c|c|c||}
	\hhline{|=================|}
	$\RedOKh$&$ -8 $&$ -7 $&$ -6 $&$ -5 $&$ -4 $&$ -3 $&$ -2 $&$ -1 $&$  0 $&$  1 $&$  2 $&$  3 $&$  4 $&$  5 $&$  6 $&$  7 $\\
	\hhline{#=#=|=|=|=|=|=|=|=|=|=|=|=|=|=|=|=#}
	$    10 $&      &      &      &      &      &      &      &      &      &      &      &      &      &      &      &$\Z$  \\ \hline
	$     8 $&      &      &      &      &      &      &      &      &      &      &      &      &      &      &$\Z^2$&      \\ \hline
	$     6 $&      &      &      &      &      &      &      &      &      &      &      &      &      &$\Z^2$&      &      \\ \hline
	$     4 $&      &      &      &      &      &      &      &      &      &      &      &      &\s{\Z^2\o\F}&&      &      \\ \hline
	$     2 $&      &      &      &      &      &      &      &      &      &      &$\Z$  &$H_1$ &      &      &      &      \\ \hline
	$     0 $&      &      &      &      &      &      &      &      &      &$\F$  &\s{\F^2\o\G^2}&&    &      &      &      \\ \hline
	$    -2 $&      &      &      &      &      &      &      &\rnode[tr]{dom-2}{$\Z$} &$H_1$  &\rnode[tl]{cod-2}{\s{\F^2\o\G}}&&      &      &      &      &      \\ \hline
	$    -4 $&      &      &      &      &      &      &\rnode[tr]{dom-4}{$\Z^2$}&$H_2$ &\rnode[tl]{cod-4}{$\F$}  &      &      &      &      &      &      &      \\ \hline
	$    -6 $&      &      &      &      &      &$\Z^3$&$H_2$ &      &      &      &      &      &      &      &      &      \\ \hline
	$    -8 $&      &      &      &      &$\Z^4$&\s{\Z^2\o\F^2}&&    &      &      &      &      &      &      &      &      \\ \hline
	$   -10 $&      &      &      &$\Z^4$&\s{\Z\o\F}&  &      &      &      &      &      &      &      &      &      &      \\ \hline
	$   -12 $&      &      &$\Z^4$&      &      &      &      &      &      &      &      &      &      &      &      &      \\ \hline
	$   -14 $&      &$\Z^3$&      &      &      &      &      &      &      &\multicolumn{7}{c||}{%
	\smash{\lower 0.9\ht\strutbox\hbox{\normalsize {\bf Notation:} $H_n:=\Z^n\oplus\F^2\oplus\G$}}}\\
	\hhline{----------}
	$   -16 $&$\Z$  &      &      &      &      &      &      &      &      &\multicolumn{7}{l||}{}\\
	\hhline{|=================|}
\end{tabular}%
\diagarc[border=2\pslinewidth,nodesep=1pt,linecolor=blue,arcangle=20]{->}{dom-4}{cod-4}%
\diagarc[border=2\pslinewidth,nodesep=1pt,linecolor=red,linestyle=dashed,arcangle=20]{->}{dom-2}{cod-2}%
\endgroup

	\par\bigskip
	\begingroup
\Small
\renewcommand\arraystretch{1.1}%
\def\F{\Z_2}%
\def\G{\Z_3}%
\def\o{{\oplus}}
\def\s#1{\Small$#1$}
\def\t#1{\tiny$#1$}
\begin{tabular}{||r||c|c|c|c|c|c|c|c|c|c|c|c|c|c|c|c||}
	\hhline{|=================|}
	$\RedOKh$&$ -7 $&$ -6 $&$ -5 $&$ -4 $&$ -3 $&$ -2 $&$ -1 $&$  0 $&$  1 $&$  2 $&$  3 $&$  4 $&$  5 $&$  6 $&$  7 $&$  8 $\\
	\hhline{#=#=|=|=|=|=|=|=|=|=|=|=|=|=|=|=|=#}
	$    16 $&      &      &      &      &      &      &      &      &      &      &      &      &      &      &      &$\Z$  \\ \hline
	$    14 $&      &      &      &      &      &      &      &      &      &      &      &      &      &      &$\Z^3$&      \\ \hline
	$    12 $&      &      &      &      &      &      &      &      &      &      &      &      &      &$\Z^4$&      &      \\ \hline
	$    10 $&      &      &      &      &      &      &      &      &      &      &      &$\Z$  &\s{\Z^4\o\F}&&      &      \\ \hline
	$     8 $&      &      &      &      &      &      &      &      &      &      &$\Z^2$&\s{\Z^4\o\F^2}&&    &      &      \\ \hline
	$     6 $&      &      &      &      &      &      &      &      &      &$\Z^2$&$H_3$ &      &      &      &      &      \\ \hline
	$     4 $&      &      &      &      &      &      &      &      &\s{\Z^2\o\F}&$H_2$& &      &      &      &      &      \\ \hline
	$     2 $&      &      &      &      &      &      &      &$H_1$ &$H_1$ &      &      &      &      &      &      &      \\ \hline
	$     0 $&      &      &      &      &      &      &\s{\F^2\o\G^2}&$\F$&&      &      &      &      &      &      &      \\ \hline
	$    -2 $&      &      &      &      &$\Z$  &$H_1$ &      &      &      &      &      &      &      &      &      &      \\ \hline
	$    -4 $&      &      &      &$\Z^2$&$\F$  &      &      &      &      &      &      &      &      &      &      &      \\ \hline
	$    -6 $&      &      &$\Z^2$&      &      &      &      &      &      &      &      &      &      &      &      &      \\ \hline
	$    -8 $&      &$\Z^2$&      &      &      &      &      &      &      &\multicolumn{7}{c||}{%
	\smash{\lower 0.9\ht\strutbox\hbox{\normalsize {\bf Notation:} $H_n:=\Z^n\oplus\F^2\oplus\G$}}}\\
	\hhline{----------}
	$   -10 $&$\Z$  &      &      &      &      &      &      &      &      &\multicolumn{7}{l||}{}\\
	\hhline{|=================|}
\end{tabular}%
\endgroup

	\caption[Odd reduced Khovanov homology for knots $16^n_{235548}$ and
	$16^n_{635483}$]{Odd reduced Khovanov homology for knots
	$16^n_{235548}$ and $16^n_{635483}$ (upper table) and their mirror
	images (lower table).  The~solid blue arrow indicate a place where
	the~operation $\BockSqOdd$ is surjective for both knots. At the~dashed
	red arrow, $\BockSqOdd$ is non-trivial for $16^n_{235548}$, but is
	a~zero map for $16^n_{635483}$. Notice that there is no place for
	a~nontrivial $\BockSqOdd$ in the~lower table.}
	\label{tab:16}
\end{figure}

\begin{observation}
For every pair of knots from~\eqref{eq:phi-pairs} and~\eqref{eq:thetaO-pairs},
these knots are distinguished by the ranks of $\BockSqOdd$, yet $\BockSqOdd$
is trivial for their mirror images. This can be explained by
noticing that a~knot may be thin while its mirror image is not, see
the~case of $13^n_{1002}$ and $14^n_{6487}$ in Figure~\ref{tab:13-14}.
\end{observation}

\begin{observation}
Looking at Table~\ref{tab:distinguished-pairs}, it appears that $\BockEvOdd$
and $\BockOddEv$ are stronger invariants than $\BockMixed$ and
$\BockMixed^2$, while $\BockSqOdd$ is even stronger. This is indeed true with
a few notable exceptions. Among all prime non-alternating knots with at most
$16$ crossings, there are two pairs that can be distinguished by the rank of
$\BockMixed$ but not by either $\BockEvOdd$ or $\BockOddEv$:
\begin{equation}
	16^n_{129312}  \to/<->/<1em>  \overline{16}^n_{640105}\qquad
	16^n_{240722}  \to/<->/<1em>  16^n_{640105}
\end{equation}
Also, there are two pairs of knots that can be distinguished by $\BockEvOdd$
and $\BockOddEv$, but not by $\BockSqOdd$:
\begin{equation}
	\overline{16}^n_{198481}  \to/<->/<1em>  \overline{16}^n_{416282}\qquad
	\overline{16}^n_{639703}  \to/<->/<1em>  \overline{16}^n_{698630}
\end{equation}
On the other hand, the lists of distinguishable knots with up to $16$
crossings for $\BockEvOdd$ and $\BockOddEv$ indeed coincide.
\end{observation}

We also looked at all the knots that have the same even and odd \emph{reduced}
Khovanov homology, but different ranks of the reduced homological operations.
Since the width of the reduced homology is always less by $1$ than the width
of the corresponding unreduced homology, and since the even reduced Khovanov
homology has, in general, very little torsion, it is natural to expect that
the reduced homological operations would be not as strong as their unreduced
counterparts. This is indeed the case.

\begin{observation}
Among all prime non-alternating knots with at most $16$ crossings that have
the same even and odd reduced Khovanov homology, there are only $4$ pairs that
can be distinguished by ranks of either $\RedBockMixed$, or $\RedBockOddEv$, or
$\RedBockEvOdd$, or $\RedBockSqOdd$:
\begin{equation}
\begin{aligned}
	\overline{16}^n_{209296} &\to/<->/<1em>  \overline{16}^n_{699643} &\qquad
						 16^n_{235548} &\to/<->/<1em>  16^n_{635483} \\
	\overline{16}^n_{485898} &\to/<->/<1em>  16^n_{543682} &\qquad
	\overline{16}^n_{910482} &\to/<->/<1em>  16^n_{919988}
\end{aligned}
\end{equation}
$\RedBockMixed$ distinguishes their mirror images as well, but no other reduced
homological operation does. See Figure~\ref{tab:16} for an example.
\end{observation}

We finish this paper with a~few more conjectures of various levels of plausibility.

\begin{conjecture}
	Every two knots that are distinguished by the~ranks of $\BockEvOdd$ are also distinguished by those of $\BockOddEv$ and vise versa. The~same is also true for $\RedBockEvOdd$ and $\RedBockOddEv$.
\end{conjecture}

\begin{conjecture}
	$\UKh(L)$ is even stronger invariant than ranks of all the~homological operations discussed in this paper.
\end{conjecture}

\end{document}